\documentclass[10pt, reqno]{amsart}
\usepackage{a4wide}
\usepackage[english, activeacute]{babel}
\usepackage{pifont}
\usepackage{amsmath,amsthm,amsxtra}
\usepackage{epsfig}
\usepackage{amssymb}
\usepackage{latexsym}
\usepackage{amsfonts}

\usepackage{enumerate}
\usepackage{xparse}
\usepackage{tabularx} 
\usepackage{multicol}
\usepackage{color, xcolor}
\usepackage{hyperref}
\usepackage{float}
\usepackage{graphicx}
\usepackage{subcaption}
\usepackage{listings}

\newcommand{\R}{\mathbb{R}}

\newcommand{\al}{\alpha}
\newcommand{\bt}{\beta}
\newcommand{\ga}{\gamma}

\newcommand{\sech}{\operatorname{sech}}

\newtheorem{thm}{Theorem}[section]
\newtheorem{cor}[thm]{Corollary}
\newtheorem{lemma}[thm]{Lemma}
\newtheorem{prop}[thm]{Proposition}

\newtheorem{defn}[thm]{Definition}

\theoremstyle{remark}
\newtheorem{rem}{Remark}[section]

\newcommand{\be}{\begin{equation}}
\newcommand{\ee}{\end{equation}}
\newcommand{\bp}{\begin{proof}}
\newcommand{\ep}{\end{proof}}
\newcommand{\bel}{\begin{equation}\label}
\newcommand{\eeq}{\end{equation}}
\newcommand{\bea}{\begin{eqnarray}}
\newcommand{\eea}{\end{eqnarray}}
\newcommand{\bee}{\begin{eqnarray*}}
\newcommand{\eee}{\end{eqnarray*}}
\newcommand{\ben}{\begin{enumerate}}
\newcommand{\een}{\end{enumerate}}

\date{\today}

\usepackage[english]{babel}

\title[Stability of breathers on the half-line]{STABILITY OF MKDV BREATHERS ON THE HALF-LINE}

\author[M.A. Alejo]{Miguel A. Alejo}
\address{Departamento de Matem\'aticas, Universidad de C\'ordoba, Spain.}
\email{malejo@uco.es}

\author[M. Cavalcante]{M\'arcio Cavalcante}
\address{Instituto de Matem\'atica, Universidade Federal de Alagoas, Macei\'o-Brazil.}
\email{marcio.melo@im.ufal.br}

\author[A.J. Corcho]{ Ad\'an J. Corcho}
\address{Instituto de Matem\'atica, Universidade Federal do Rio de Janeiro,  Brazil.}
\email{adan@im.ufrj.br}

\thanks{Cavalcante was supported by CNPq 310271/2021-5 and CAPES-MATHAMSUD
	88887.368708/2019-00}
\thanks{A.J. Corcho was  supported by  CNPq grant no.~307616/2020-7, Brazil and  Carolina Fundation Grants 2020-2021, Spain.}



\begin{document}

\begin{abstract}

In this paper we study the stability problem for mKdV breathers on the left half-line.  We are able to show that leftwards moving  breathers, 
initially located far away from the origin, are strongly stable for the problem posed on the left half-line, when assuming homogeneous boundary conditions. The
proof involves a Lyapunov functional which is almost conserved by the mKdV flow once we control some boundary terms which
naturally arise.
\end{abstract}

\maketitle
\numberwithin{equation}{section}

\bigskip

\section{Introduction}

\subsection{Setting of the problem} This paper deals with the nonlinear stability of \emph{breathers} of the focusing modified
Korteweg-de Vries (mKdV) equation \cite{wadati} posed on the \emph{left half-line} $\R^{-}:= (-\infty,0)$: 

\begin{equation}\label{mkdv}
\partial_tu+\partial_x(\partial_x^2 u+u^3)=0,\quad u(x,t)\in\R,\quad (x,t)\in\mathbb R^{-}\times(0,T).
\end{equation}

\medskip
The focusing\footnote{Focusing or defocusing means $\pm u^3$ respectively in the equation.} mKdV equation \eqref{mkdv} in the whole real line $\R$,  is an integrable and canonical non-linear dispersive equation, originally describing  
shallow water wave dynamics \cite{Whit}, and therefore appearing as a good approximation of different  physical problems. A few examples are 
the motion of the curvature 
of some geometric fluxes \cite{GoPe,Na1,Na2}, vortex patches, ferromagnetic vortices \cite{Wex},  traffic models, anharmonic lattices, hyperbolic surfaces, among others. As a consequence of its integrability in $\R$, it is possible to get explicit solutions. For instance, the simplest one is the (real-valued) \emph{mKdV soliton} solution which, to be more precise, has the form
\begin{equation}
	u(x,t)=\widetilde{Q}_c(x-ct-x_0),\quad \widetilde{Q}_c(s):=\sqrt{c}\widetilde{Q}(\sqrt{c}s),\quad c>0,~x_0\in\R,
\end{equation}
where
\begin{equation}\label{soliton}
	\widetilde{Q}(s)= \frac{\sqrt{2}}{\cosh(s)}=2\sqrt{2}\partial_s[\arctan(e^s)],
\end{equation}
\noindent
with  $c$  the propagation speed of the wave.  The real-line soliton $\widetilde{Q}_c$ satisfies the following ``boundary value problem'' (BVP) on $\R$,
\begin{equation}\label{edosol}
	\begin{cases}
		\widetilde{Q}_c''-c\widetilde{Q}_c+\widetilde{Q}_c^3=0,& x\in \R,\\
		\lim\limits_{x\to \pm \infty}\widetilde{Q}(x)=0,
	\end{cases}
\end{equation}
and it is the unique positive $H^1(\R)$-solution of \eqref{mkdv} up to translations in space. 

\medskip
A Cauchy theory for the initial value problem (IVP) for the \emph{focusing} mKdV posed on the real axis,
\begin{equation}\label{ivpmkdv}
	\begin{cases}
		\partial_tu+\partial_x(\partial_x^2 u+u^3) =0 ,& (x,t)\in\R\times\R, \medskip \\
		u(x,0)=u_0(x),                                   & x\in\R,
	\end{cases}
\end{equation}
has been extensively studied in the last years. In the case of real-valued initial data, the IVP for \eqref{ivpmkdv}
 is globally well posed
for initial data in $H^s(\R)$ for any $s > 1/4$; see  \cite{KPV} and  \cite{CKSTT}. Moreover, the (real-valued) flow map is not 
uniformly continuous if $s < 1/4$  (see \cite{KPV2}). This was proved by using a special family of solutions of \eqref{ivpmkdv} 
called \emph{breathers}, and discovered by Wadati \cite{wadati}. Explicitly, the mKdV breather is defined as follows. 

\begin{defn}[See e.g. \cite{wadati,La}]\label{defbreather} 
Let $\al, \bt  >0$ and $x_1,x_2\in \R$ be fixed parameters. The focusing mKdV breather is a smooth solution of \eqref{ivpmkdv} given  by the formula

\bea\label{breather}
\begin{split}	
\widetilde B_{\al, \bt}(x,t;x_1,x_2) 
    &:=  2\sqrt{2} \partial_x \bigg[\arctan \Big( \frac{\bt}{\al}\frac{\sin(\al y_1)}{\cosh(\bt y_2}\Big)\bigg]\\
    & =2 \sqrt{2} \bt \sech (y_2) \bigg[\frac{\cos (\al y_1) - (\bt /\al) \sin (\al y_1) \tanh (\bt y_2)} {1 +(\bt/\al)^2 \sin^2 (\al y_1)\sech^2 (\bt y_2)}\bigg],			
\end{split}
\eea
where
\begin{align}
&y_1 := x+ \delta t + x_1, \quad y_2 : = x+ \ga t + x_2\label{y1y2}\\ 
\intertext{and}
&\delta := \al^2 -3\bt^2, \quad  \ga := 3\al^2 -\bt^2.\label{deltagamma}
\end{align}
\end{defn}
\noindent{
Observe that this wave like solution of \eqref{ivpmkdv} is periodic in variable $y_1$ and localized in variable $y_2$. Also, note that  $\gamma\neq\delta$ for any $\alpha, \beta  \neq 0$, which implies  that the 
\textit{traveling wave arguments}\footnote{Assuming the simplest case $x_1=x_2=0$.}:
$$y_1=x+\delta t\quad \text{and}\quad  y_2=x+\gamma t$$
are always different.} Currently, $\beta$ and $\alpha$ are called  {\it amplitude} and {\it frequency} parameters of the breather, 
and $-\gamma$ will be the \emph{velocity} of the mKdV breather solution \eqref{breather}. Note that this 
corresponds to the speed of the {\it$\sech$} envelope of the breather profile, dragging to the left or to the right (depending on its sign) 
the corresponding inner oscillations of the breather. In \cite{AM} it was proved that breather solutions of the focusing mKdV equation \eqref{ivpmkdv} in $\R$ are actually globally stable in a natural $H^2$ topology. In the proof the authors introduced a new Lyapunov functional, 
at the $H^2$ level, which allowed to describe the dynamics of small perturbations, including oscillations induced by the periodicity 
of the solution, as well as a direct control of the corresponding instability modes. In particular, degenerate directions were controlled using low-regularity conservation laws. Finally,  
we point out that in \cite{conjecture} the soliton resolution for the focusing mKdV equation on the real line $\R$, was established for initial conditions in some weighted Sobolev spaces, where one should realize that general solution to the focusing mKdV
will consist of solitons moving to the right, breathers traveling to both directions and a radiation
term.  Moreover, the authors obtained the asymptotic stability of nonlinear structures involving solitons and breathers.

\medskip
Note that (regular) breather solutions only appear in { some particular  PDEs. For instance, in gKdV models,  they only arise in the mKdV \eqref{mkdv} but,  they do not appear in the KdV case}, as it was recently 
proved \cite{MP}. Therefore, proffiting its existence in the mKdV model \eqref{ivpmkdv}, our main aim in this work will be to approach the stability analysis of 
{\it focusing} mKdV breathers in the \emph{left half-line} $\R^{-}$. As a direct consequence, we present two main contributions: firstly, we go a step further, 
in comparison with \cite{CM}, where the stability analysis for simpler solutions, like KdV solitons in the half-line was presented. Secondly, we extend previous 
stability results of mKdV breathers in the real line $\R$ (see \cite{AM}), by adapting these techniques to the case of boundary conditions as it corresponds to a $\R^{-}$ domain, which is more realistic case for experimental purposes.

\medskip
In this work, we consider the mKdV equation on the left half-line and we will deal with mKdV breather solutions \eqref{breather} \emph{moving leftwards} in space, 
and therefore when its velocity $-\gamma<0$ or equivalently, from \eqref{deltagamma}, when $\beta<\sqrt{3}\alpha$. In this situation we can impose two boundary conditions 
for the IBVP  \eqref{IBVP_left}.

\begin{figure}[h]\label{fig-breather-left}
\begin{center}
\includegraphics[scale=0.65]{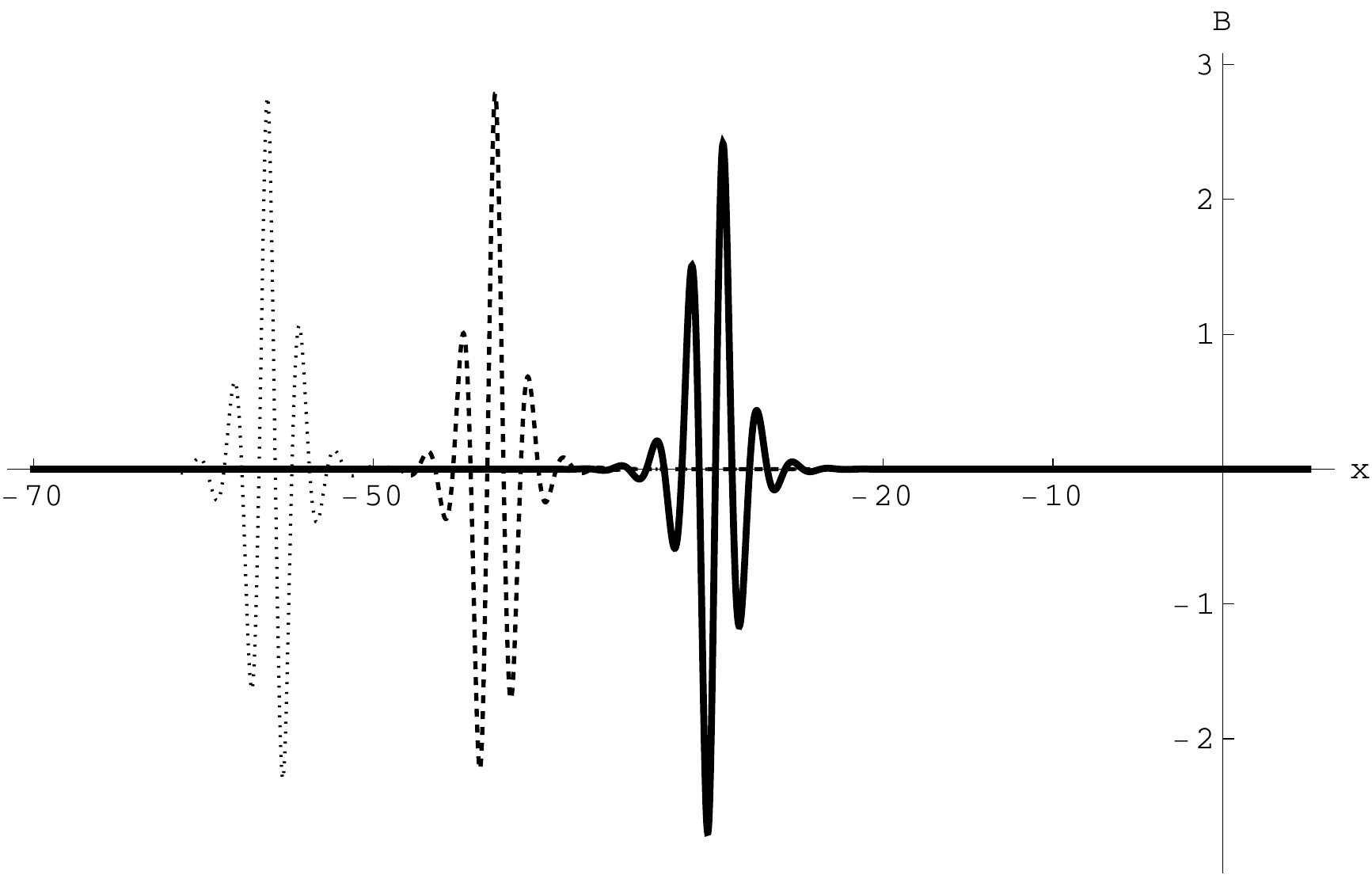}
\end{center}
\caption{The evolution of the mKdV breather \eqref{breather},  with  $\alpha =3,~\beta=1,~x_1=0$ and $x_2=30$ at times $0,~0.5$ and $1$ (full, dashed, dotted lines, respectively).
In this case, $- \ga =\bt^2 -  3\al^2<0$ and hence the breather moves leftwards.}
\end{figure}

\medskip 
It remains as an interesting open problem to study the stability properties of these mKdV breathers on the right hand side $\R^+=(0,+\infty)$. 
In fact, a few differences  with respect to the left hand side arise in that case. For instance, the case of rightwards moving breathers, implies 
that $\beta>\sqrt{3}\alpha$. Unfortunately, in this situation, we can  not  impose a second boundary 
condition of the corresponding IBVP. This fact prevent us from constructing a suitable Lyapunov functional, 
almost conserved and well defined on $H^2(\R^+)$ (See Remark \ref{right}).

\medskip
From another point of view, many physical problems naturally arise as initial boundary value problems (IBVP), 
because of the local character of the corresponding phenomenon \cite{Zabusky}. However, the IBVP for the mKdV 
equation has been considerably less studied than the corresponding IVP \eqref{ivpmkdv}. For example, there are at 
least two interesting IBVP for mKdV still in unbounded domains: the one posed on the right half-line, and a second one 
posed on the left portion of the line, which we consider in this work.

\subsection{Unbounded initial boundary value problems} The IBVP for the \emph{focusing} mKdV equation posed on the {\bf left} half-line 
is the following: for $\mathbb R^-:= (-\infty,0)$ and $T>0$, {we look for solutions $u$ of the model}
\begin{equation}\label{IBVP_left}
\begin{cases}
	\partial_tu+\partial_x(\partial_x^2 u+u^3)=0  ,  & (x,t)\in\mathbb R^-\times(0,T),\\
	u(x,0)=u_0(x),                                   & x\in\mathbb R^-,\\
	u(0,t)=f(t),                                     & t\in(0,T),\\
	\partial_xu(0,t)=f_1(t),                         & t\in(0,T).\\
\end{cases}
\end{equation}

In the recent literature, the mathematical study of IBVP \eqref{IBVP_left} is usually considered in the following setting 
\begin{equation}\label{setting_B}
	(u_0,f,f_1)\in H^{s}(\R^-)\times H^{(s+1)/3}(\R^+)\times H^{s/3}(\R^+).
\end{equation}
These assumptions are in some sense sharp because of the following localized smoothing effect for the linear evolution \cite{KPV}
\begin{align*}
	&\|\psi(t) e^{-t\partial_x^3}\phi(x)\|_{\mathcal{C}\big(\mathbb{R}_x;\; H^{(s+1)/3}(\mathbb{R}_t)\big)}\lesssim \|\phi\|_{H^s(\mathbb{R})},\\
	\intertext{and}
	&\|\psi(t) \partial_xe^{-t\partial_x^3}\phi(x)\|_{\mathcal{C}\big(\mathbb{R}_x;\; H^{s/3}(\mathbb{R}_t)\big)}\lesssim \|\phi\|_{H^s(\mathbb{R})},
\end{align*}
where $\psi(t)$ is a smooth cutoff function and $e^{-t\partial_x^3}$, denoting the linear homogeneous solution group on $\mathbb{R}$ associated
to the linear  part of the  equation in \eqref{IBVP_left}. Therefore, and hereafter, we will  follow the setting \eqref{setting_B}.

\medskip 
Other classical IBVP is the mKdV on the right half-line given by
\begin{equation}\label{IBVP_right}
\begin{cases}
\partial_tu+\partial_x(\partial_x^2 u+u^3)=0  ,  & (x,t)\in\mathbb R^+\times(0,T),\\
u(x,0)=u_0(x),                                   & x\in\mathbb R^+,\\
u(0,t)=f(t),                                     & t\in(0,T).
\end{cases}
\end{equation}
The presence of one boundary condition in  \eqref{IBVP_right} versus two boundary conditions in the left half-line problem 
\eqref{IBVP_left} for the KdV-component of the system is justified in \cite{Holmer}. The local well-posedness was considered in \cite{CK} on the Sobolev Spaces $H^{\frac14}(\R^+)$. 
It was recently shown in \cite{CM} that solitons initially posed  far away from the origin are strongly stable for the problem posed 
on the right half-line, assuming homogeneous boundary conditions. The proof of this stability result involved the construction of two almost conserved 
quantities adapted to the evolution of the KdV soliton, in the particular case of the half-line.

\medskip
With respect to previous advances, Faminskii showed global well-posedness for the following IBVP associated to the classical KdV equation 
(see \cite{Fa4}):
\begin{equation}\label{IBVP_left_classical}
\begin{cases}
\partial_tu+\partial_x(\partial_x^2 u+u^2)=0,  & (x,t)\in\mathbb R^-\times(0,T),\\
u(x,0)=u_0(x),                                 & x\in\mathbb R^-,\\
u(0,t)=f(t),                                   & t\in(0,T),\\
\partial_xu(0,t)=f_1(t),                       & t\in(0,T).
\end{cases}
\end{equation}

In the current work, we consider the solution $u$ posed on the space 

\begin{equation}\label{space-functions}
u\in \mathcal{C}\big(\R^+;H^2(\R^-)\big) \quad \text{and}\quad  \partial_{x}^{j}u\in\mathcal{C}\big(\mathbb{R}_{x}^{+} ; H^{(3-j) / 3}(0, T)\big)~~ \text {for}~~\  j=0,1,2,3.
\end{equation}

\medskip 
\begin{rem}[Well-posedness]\label{rem-wp}
Concerning the well-posedness theory for the IBVP \eqref{IBVP_left} at the level $H^2(\R^-)$, we remark the following:
 
\medskip 
\begin{enumerate}[(a)]
 \item (\textit{Local Theory}). The approach used by Faminskii in \cite{Fa4}  to solve a similar problem by considering the 
 quadratic nonlinearity can be applied to our current problem to get a local theory. In fact, local solutions in $\mathcal{C}([0, T];\, H^2(\R^-))$ 
 for the IBVP \eqref{IBVP_left},  with conditions \eqref{setting_B} at the regularity level $s=2$,  
 can be constructed by using the contraction principle. In such a case, the main difficulty is to get the fundamental 
 trilinear estimate needed to solve \eqref{IBVP_left} on the modified Bourgain spaces adapted to the corresponding problem posed on the half-line. 
 This is a technical argument and it can be obtained by using similar ideas contained in  \cite{CC},  where the modified Kawahara equation with cubic nonlinearities was studied. There, the key point was to obtain the corresponding trilinear estimates 
 (see Theorem 1.1  in \cite{CC}).
  
\medskip
\item (\textit{Global Theory}). Local solutions obtained in (a) can be extended globally in time from apriori 
estimates presented in Section \ref{Section_Fun-Est} (see Corollary \ref{cor-gwp}).
\end{enumerate}
\end{rem}

\subsection{Main result} 
We consider a breather solution on the left half-line as the restriction on $\R^{-}$ of 
classical breathers posed on the whole line 
\eqref{breather}, i.e.
\begin{equation}\label{semibreather}
B_{\alpha,\beta}=\widetilde B_{\alpha,\beta}\bigg|_{\R^{-}}.
\end{equation}

We highlight that the above breather on the left half-line is not an exact solution for the IBVP \eqref{IBVP_left}, except for very particular
boundary conditions $f(t)$ and $f_1(t)$. More precisely, restricted breathers $B=B(x,t;x_1,x_2)$ induce the natural traces given by

\begin{equation}\label{Breathers-lhl-traces}
f(t)=B(x=0,t;x_1,x_2)\quad  \text{and}\quad f_1(t)=\partial_xB(x=0,t;x_1,x_2).  
\end{equation}

\medskip
In this work we will prove that any classical mKdV breather solution,  restricted to the left half-line $\R^-$ \eqref{semibreather}, 
and placed far enough from the origin $x = 0,$ is stable in  $H^2(\R^-)$ under perturbations 
that preserve the zero boundary conditions.  More precisely, 
we prove the following:

\begin{thm}[Nonlinear $H^2$ stability of mKdV breathers on the left half-line]\label{teorema} 
	Let $\alpha, \beta>0$, and $B_{\alpha,\beta}$ a restricted breather \eqref{semibreather}. Assuming that $\beta \leq \alpha$ (breathers moving leftwards), there exist parameters $\eta_{0}, A_{0}$ and $L_0$, depending on $\alpha$ and $\beta$, such that for all  $L>L_0$ and  $\eta \in(0, \eta_0)$ 
	the following holds: consider $u_{0} \in H^{2}(\mathbb{R}^-)$ such that
	\begin{equation}\label{initial}
	\left\|u_{0}-B_{\alpha, \beta}(\cdot, 0 ; 0,L)\right\|_{H^{2}(\mathbb{R}^-)} \leq \eta.
	\end{equation}
	Then there exist  continuous functions $\rho_{1}(t), \rho_{2}(t) \in \mathbb{R}$ such that the solution $u(\cdot, t)$ of the IBVP 
	\eqref{IBVP_left} with initial data $u_{0}$ and homogeneous boundary conditions $f(t)=f_1(t)\equiv 0$, satisfies
	\begin{align}
	&\sup _{t \in \mathbb{R}}\left\|u(t)-B_{\alpha, \beta}\left(\cdot, t ; \rho_{1}(t), \rho_{2}(t) + L\right)\right\|_{H^{2}(\mathbb{R}^-)} \leq A_{0} \eta+Ke^{-\beta L} \label{teorema-a}
	\end{align}
for some constant $K>0$.
\end{thm}

This result shows that leftwards moving breathers posed initially far away from the origin are strongly stable for
the IVBP  problem \eqref{IBVP_left} posed on the left half-line, assuming homogeneous boundary conditions.

\medskip
Our proof involves an almost conserved Lyapunov functional, for which we have to control some boundary terms. 
In addition, we have some error contributions that appear because the restricted breather \eqref{semibreather} is not 
an exact solution for the initial boundary value problem \eqref{IBVP_left}.

\medskip
\begin{rem} Some points deserve to be enlighted:

\medskip 
\begin{enumerate}[(a)]
\item(\textit{On the zero boundary condition}).
 Note that conditions $u(x=0,t)=u_x(x=0,t)=0$ are assumed to avoid  bad trace 
 higher order functions on the energy identities, which are the fundamental ingredients to 
 construct the almost conserved Lyapunov functional.  The case with non-homogeneous boundary conditions raises 
 as an interesting open  problem.

\medskip 
\item (\textit{Right half-line}).
The case of the IBVP on the right half-line remains as a challenging open problem. This problem imposes 
several new conditions with respect to the left half-line case, as for instance, that the breather speed 
$-\gamma>0$ or that we can not impose a second boundary condition to the corresponding IBVP.

\medskip 
\item (\textit{Applications}). 
 We think that the developed techniques and ideas presented in this work can be applied, with minor changes but with 
 more involved computations, to the Gardner equation posed on the left half-line
 \be\label{GE}
 w_t +(w_{xx} +3\mu w^2+w^3)_x=0, \quad \mu\in\R\backslash\{0\},\qquad w(x,t) \in \R,\quad (x,t)\in \mathbb R^{-}\times(0,T).
 \ee
 This model can be thought as a perturbed {\it focusing} mKdV equation, by a small parameter $\mu\in\R\backslash\{0\}$ 
 controlling the strength of the quadratic nonlinear part or KdV term $w^2$. The Gardner equation \eqref{GE} also bears breather solutions, 
 and they can be interpreted as perturbed mKdV breathers. See \cite{Ale1} for further details. 
\end{enumerate}
\end{rem}

\subsection{Organization of this paper} 
After some preliminaries in Section \ref{Section_2}, we show restricted functionals to the left half-line in Section \ref{Section_Fun-Est}. Afterwards, in Section \ref{ProofTheo} we prove the main Theorem \ref{teorema}. Finally in Appendix \ref{proof-lemma} and \ref{Appboundary} we explicitly prove some technical previous results. 

\subsection{Acknowledgments}
We would like to thank to the Departamento de Matem\'aticas, Universidad de C\'ordoba, Spain, where part of this work was done. Third author also thanks 
\emph{Fundaci\'on Carolina} for its funding support while this work was in preparation.

\section{Preliminaries}\label{Section_2}
In this section we summarize some useful facts obtained in \cite{AM} about  breather profiles on $\R$.

\medskip 
\begin{lemma}\label{equations} The mKdV breather $\widetilde{B}:=\widetilde{B}_{\alpha, \beta}$ \eqref{breather} 
satisfies the following properties:
 
 \medskip 
 \begin{enumerate}[(i)]
 \item  $\widetilde{B}=\frak{B}_{x}$, with $\frak{B}=\frak{B}_{\alpha, \beta}$ given by the smooth $L^{\infty}$-function,
 	\begin{equation}\label{breather-primitive}
 	\frak{B}(x,t):=2 \sqrt{2} \arctan \left(\frac{\beta}{\alpha} \frac{\sin \left(\alpha y_{1}\right)}{\cosh \left(\beta y_{2}\right)}\right).
 	\end{equation}
 	
 \medskip 
 \item For any fixed $t \in \mathbb{R}$, we have $\frak{B}_{t}$ well-defined in the Schwartz class, satisfying
 	\begin{equation}\label{Bt}
 	\widetilde B_{x x}+\frak{B}_{t}+\widetilde B^{3}=0.
 	\end{equation}
 	
 	\medskip 
    \item  For all $t \in \mathbb{R}$, $\widetilde{B}$ satisfies 
    \begin{equation}\label{Bxt}
    \widetilde B_{x t}+2\left(\mathcal{M}_{\alpha, \beta}\right)_{t} \widetilde B=2\left(\beta^{2}-\alpha^{2}\right) \frak{B}_{t}+\left(\alpha^{2}+\beta^{2}\right)^{2} \widetilde B,
    \end{equation}
where
\begin{multline}
	\mathcal{M}_{\alpha, \beta}(x,t):=\frac{1}{2} \int_{-\infty}^{x} \widetilde B_{\alpha, \beta}^{2}\left(s,t ; x_{1}, x_{2}\right) d s\\
	=\frac{2 \beta\left[\alpha^{2}+\beta^{2}+\alpha \beta \sin \left(2 \alpha y_{1}\right)-\beta^{2} \cos \left(2 \alpha y_{1}\right)+\alpha^{2}\left(\sinh \left(2 \beta y_{2}\right)+\cosh \left(2 \beta y_{2}\right)\right)\right]}{\alpha^{2}+\beta^{2}+\alpha^{2} \cosh \left(2 \beta y_{2}\right)-\beta^{2} \cos \left(2 \alpha y_{1}\right)}.
\end{multline}

\medskip 	 
\item Also, for all $t \in \R,$ $\widetilde B$ satisfies the nonlinear stationary equation
\begin{equation}\label{ODEBreather}
G[\widetilde B]:=\widetilde B_{(4 x)}-2\left(\beta^{2}-\alpha^{2}\right)\left(\widetilde B_{x x}+\widetilde B^{3}\right)+\left(\alpha^{2}+\beta^{2}\right)^{2} \widetilde B+5 \widetilde B \widetilde B_{x}^{2}+5 \widetilde B^{2} \widetilde B_{x x}+\frac{3}{2} \widetilde B^{5}=0. 
\end{equation}
\end{enumerate}
\end{lemma}

\medskip 
Another important ingredient  defined in \cite{AM} is the  fourth order linear operator
\begin{multline}\label{operator-L}
\mathcal{L}[z](x ; t):= z_{(4 x)}(x)-2\left(\beta^{2}-\alpha^{2}\right) z_{x x}(x)+\left(\alpha^{2}+\beta^{2}\right)^{2} z(x)+5 \widetilde{B}^{2} z_{x x}(x)+10 \widetilde B \widetilde B_{x} z_{x}(x) \\
+\Big[5 \widetilde B_{x}^{2}+10 \widetilde B \widetilde{B}_{x x}+\frac{15}{2} \widetilde{B}^{4}-6\left(\beta^{2}-\alpha^{2}\right) \widetilde{B}^{2}\Big] z(x),
\end{multline}

\medskip 
\noindent
and its associated quadratic form:
\begin{equation}\label{quadratic}
\begin{aligned}
\widetilde{\mathcal{Q}}[z]&:=\int_{\mathbb{R}} z \mathcal{L}[z]\\
& = \int_{\mathbb{R}} z_{x x}^{2}+2\left(\beta^{2}-\alpha^{2}\right) \int_{\mathbb{R}} z_{x}^{2}+\left(\alpha^{2}+\beta^{2}\right)^{2} \int_{\mathbb{R}} z^{2}-5 \int_{\mathbb{R}} B^{2} z_{x}^{2}\\
&\hspace{2cm} + 5 \int_{\mathbb{R}} B_{x}^{2} z^{2}+10 \int_{\mathbb{R}} B B_{x x} z^{2}+\frac{15}{2} \int_{\mathbb{R}} B^{4} z^{2}-6\left(\beta^{2}
-\alpha^{2}\right) \int_{\mathbb{R}} B^{2} z^{2}.
\end{aligned}
\end{equation}

\medskip 
Now we introduce two important directions associated to spatial translations. Let $\widetilde{B}_{\alpha, \beta}$ as in \eqref{breather}. We define
\begin{equation}\label{B1-B2-tilde}
\widetilde B_{1}\left(x,t ; x_{1}, x_{2}\right):=\partial_{x_{1}} \widetilde{B}_{\alpha, \beta}\left(x,t ; x_{1}, x_{2}\right) \quad \text{and} 
\quad  \widetilde B_{2}\left(x,t ; x_{1}, x_{2}\right):=\partial_{x_{2}} \widetilde{B}_{\alpha, \beta}\left(x,t ; x_{1}, x_{2}\right).
\end{equation}
It is clear that, for all $t \in \mathbb{R},$ $\alpha, \beta>0$ and $x_{1}, x_{2} \in \mathbb{R},$ both $\widetilde B_{1}$ and $ \widetilde B_{2}$ are real-valued functions in the Schwartz class, exponentially decreasing in space. Moreover, it is not difficult to see that they are linearly independent as functions of the $x$ -variable, for all time $t$ fixed.

\medskip
\noindent
The following result in \cite{AM} will be useful:
\begin{prop}\label{coec}
 Let $\widetilde B=\widetilde{B}_{\alpha, \beta}$ be any $m K d V$ breather, and $\widetilde B_{1}, \widetilde B_{2}$ the corresponding kernel of the associated operator $\mathcal{L}$. There exists $\mu_{0}>0$, depending only on $\alpha, \beta$, such that, for any $z \in H^{2}(\mathbb{R})$ satisfying
	$$
	\int_{\mathbb{R}} \widetilde B_{1} z=\int_{\mathbb{R}} \widetilde B_{2} z=0
	$$
	one has
	$$
	{\widetilde{\mathcal{Q}}[z]} \geq \mu_{0}\|z\|_{H^{2}(\mathbb{R})}^{2}-\frac{1}{\mu_{0}}\left(\int_{\mathbb{R}} z \widetilde B\right)^{2}.
	$$
\end{prop}

In what follows we denote
\begin{align}
&B(x,t; x_1, x_2)=B_{\alpha, \beta}(x,t; x_1, x_2),\label{breather-simp-not}\\
&B_{j}(x,t ; x_1, x_{2}):=\partial_{x_{_j}} B_{\alpha, \beta}(x,t ; x_1, x_2),\; j=1,2,\label{B1-B2}
\end{align}
with $B_{\alpha, \beta}$ defined in \eqref{semibreather}, in order to simplify  future computations.

\section{Almost conserved Lyapunov functional}\label{Section_Fun-Est}

In this section we will define a suitable Lyapunov functional in the spirit of \cite{AM}, keeping in mind the boundary terms. 

\medskip
The following functionals (obtained from the first three conserved quantities of \eqref{mkdv}) 
will be important to understand the dynamics of the solutions $u(\cdot, t)$ of the  IBVP \eqref{IBVP_left} close to breathers,
\begin{equation}\label{Mass}
	M[u](t):=\frac12\int_{\R^-}u^2(x,t)dx,\quad  {\it \text{(mass)}}
\end{equation}
\begin{equation}\label{Energy}
	E[u](t):=\int_{\R^-}\Big( \frac12 u_x^2(x,t) -\frac14 u^4(x,t)\Big)dx,\quad {\it(\text{energy})}
\end{equation}
and
\begin{equation}\label{Second-Energy}
	F[u](t):=\int_{\R_-} \Big( \frac{1}{2}u^2_{x x}(x, t) -\frac{5}{2} u^{2}(x, t) u_{x}^{2}(x,t) +\frac{1}{4} u^{6}(x, t) \Big)dx,\quad   {\it(\text{second order energy})}
\end{equation}
which are well-defined for solutions in $\mathcal{C}(\R;H^2(\R^-))$. 

\medskip 
Before presenting some key functional estimates, we define the following nonlinear terms 
which will appear in the computations. Explicitly,  in the current context of half-line domains, they arise as additional factors associated to boundary terms. Namely
\begin{equation}\label{trace-terms-M}
	\tau_M(x,t):= \tfrac{1}{2}u^2_x -u_{xx}u - \tfrac{3}{4}u^4,
\end{equation}

\begin{equation}\label{trace-terms-E}
	\tau_E(x,t):= \tfrac{1}{2}u^6 + u^3u_{xx} + \tfrac{1}{2}u^2_{xx}  - u_{xxx}u_x - 3u^2u_x^2 	
\end{equation}
and 
\begin{multline}\label{trace-terms-F}
	\tau_F(x,t):=  -u_t(u^3)_x - \tfrac{9}{2}u^4u_x^2 + \tfrac{1}{2}u_{xxx}^2 + u_{xx}u_{xt}\\ -u^2u^2_{xx} -2u_tu^2u_x + \tfrac{3}{4}u^4u_x^2 -\tfrac{1}{4}u_x^4 + uu_x^2u_{xx}-\tfrac{3}{2}u^5u_{xx} -\tfrac{9}{16}u^8.
\end{multline}

Note that the above trace terms  $\tau_M(0,t)$, $\tau_E(0,t)$ and $\tau_F(0,t)$ are well defined for solutions $u$ 
on the space 

$$\mathcal{U}_T(\R^-):=\Big\{ u\in \mathcal{C}(\R^+;H^2(\R^-)): \partial_{x}^{j} u \in\mathcal{C}\left(\mathbb{R}_{x}^{-} ; H^{(3-j) / 3}(0, T)\right),\; j=0,1,2,3\Big\}.$$

\noindent
Moreover, note the following:
\begin{lemma}\label{lemma}
	Let $u=u(x,t)$  be the solution of the IBVP \eqref{IBVP_left} with
	initial data $u_0 \in H^2(\R^-)$. Then, the following identities are satisfied:
	\begin{equation}\label{mass-ident}
		M[u](t)= M[u_0] + \int_0^t\tau_M(0,s),
	\end{equation}
	\begin{equation}\label{energy-ident}
		E[u](t) = E[u_0] +\int_0^t\tau_E(0,s)ds,
	\end{equation}
	and 		
	\begin{equation}\label{second-energy-ident}
		F[u](t) = F[u_0] + \int_0^t\tau_F(0, s)ds,
	\end{equation}
	for all $t\ge 0$.  Moreover, under  homogeneous boundary conditions 
	\begin{equation}\label{lemma-h-condit}
		u(0,t)=0\quad \text{and}\quad u_x(0,t)=0
	\end{equation}
	we have 
	\begin{equation}\label{mass-ident-h}
		M[u](t) = M[u_0],
	\end{equation}
	\begin{equation}\label{energy-ident-h}
		E[u](t) \ge  E[u_0],
	\end{equation}
	and 		
	\begin{equation}\label{second-energy-ident-h}
		F[u](t) = F[u_0],
	\end{equation}
	for all $t\ge 0$.
\end{lemma}

\begin{proof}
See the proof in Appendix \ref{proof-lemma}.
\end{proof}

\begin{cor}\label{cor-gwp}
The local solution of IBVP \eqref{IBVP_left} with initial data $u_0 \in H^2(\R^-)$ and homogeneous boundary conditions
$$u(x,0) = u_x(0,t) = 0, \quad t\in [0, T),$$
described in Remark \ref{rem-wp}-(a), can be extended globally in time. 
\end{cor}

\begin{proof}
The idea is to derive an apriori estimate of the norm
$$\|u(\cdot, t)\|_{L^2(\R^-)} + \|u_{xx}(\cdot, t)\|_{L^2(\R^-)},$$
by using the conservation of the functionals \eqref{mass-ident-h} and \eqref{second-energy-ident-h}.

\medskip
In view of the conservation \eqref{mass-ident-h} we only need to get a control of the $\|u_{xx}(\cdot, t)\|_{L^2(\R^-)}$. To proceed, we first   note that from \eqref{second-energy-ident-h} we have
\begin{equation}\label{cor-gwp-1}
\int_{\R^-}\Big(\frac{1}{2}u^2_{xx} + \frac{1}{4}u^6\Big)= F[u_0] + \frac{5}{2}\int_{\R^-}u^2u^2_x.
\end{equation}
Now using integration by parts and the homogeneous boundary conditions one gets 
\begin{equation}\label{cor-gwp-2}
\psi[u]:= \frac{5}{2}\int_{\R^-}u^2u^2_x= -\frac{5}{2}\int_{\R^-}u^3u_{xx} dx -2 \psi[u].
\end{equation}
So, from \eqref{cor-gwp-2}  it follows that 
\begin{equation}\label{cor-gwp-3}
\psi[u] = -\frac{5}{6}\int_{\R^-}u^3u_{xx} \le \frac{5}{6}\|u\|^3_{L^6_x(\R^-)}\|u_{xx}\|_{L^2_x(\R^-)}.
\end{equation}

\medskip 
On the other hand by using a Gagliardo-Nirenberg inequality and  \eqref{mass-ident-h}, we get that 
\begin{equation}\label{cor-gwp-4}
\|u\|_{L^6_x(\R^-)} \lesssim  \|u_{xx}\|^{1/6}_{L^2_x(\R^-)}\|u_0\|^{5/6}_{L^2_x(\R^-)}.
\end{equation}
Thus, using \eqref{cor-gwp-4}  in \eqref{cor-gwp-3}, combined with Young's inequality, we have the estimate:
\begin{equation}\label{cor-gwp-5}
\begin{split}
\psi[u]=\frac{5}{2}\int_{\R^-}u^2u^2_x& \le C_1\|u_{xx}\|^{3/2}_{L^2_x(\R^-)}\|u_0\|^{5/2}_{L^2_x(\R^-)}\\
&\le \frac{1}{4}\|u_{xx}\|^2_{L^2_x(\R^-)} + C_2\|u_0\|^{10}_{L^2_x(\R^-)},
\end{split}
\end{equation}
for some positive constants $C_1$ and $C_2$. Finally, putting  the estimate \eqref{cor-gwp-5} in  \eqref{cor-gwp-1}  we obtain
$$
\frac{1}{4}\|u_{xx}\|^{2}_{L^2_x(\R^-)}\le \int_{\R^-}\Big(\frac{1}{4}u^2_{xx} + \frac{1}{4}u^6\Big) \le  F[u_0] +  C_2\|u_0\|^{10}_{L^2_x(\R^-)},
$$
and we have the desired apriori control for the $\|u_{xx}\|_{L^2_x(\R^-)}$. Then, the proof is finished.  
\end{proof}

\begin{rem}[About  breathers \textit{moving rightwards}]\label{right}
It is important to note that in the case of the right half-line ($\R^+$), the correspon\-ding trace terms would be 	$-\tau_M(0,t)$, $-\tau_E(0,t)$ 
and 	$-\tau_F(0,t)$. Hence, since the homogeneous boundary condition $u_x(0, t)=0$ is not allowed in the corresponding IBVP on 
$\R^{+}$, we see that, only by using the homogeneous condition  $u(x, 0)=0$, the 
term $\big(\frac{1}{4}u^4_x-u_{xx}u_{xt}\big)(0,t)$ remains in $-\tau_F(x,0)$, and this nonlinear term is difficult to control. This fact prevents us from
 building a Lyapunov functional on the right half-line. This is the main reason to not address here the case of breathers \textit{moving rightwards}.
\end{rem}

Now, we are able to introduce an almost conserved Lyapunov functional, specifically related to the breather function $B_{\alpha,\beta}$ on
$\R^-$ \eqref{semibreather}. Let  $t>0$ 
and $M[u]$, $E[u]$ and $F[u]$ the conserved quantities defined in \eqref{Mass}-\eqref{Energy}-\eqref{Second-Energy}. Based on the work \cite{AM} we define the restricted Lyapunov functional
\begin{equation}\label{Hlyapunov}
\mathcal{H}[u](t):=F[u](t)+2\left(\beta^{2}-\alpha^{2}\right) E[u](t)+\left(\alpha^{2}+\beta^{2}\right)^{2} M[u](t).
\end{equation}

\medskip 
Note that, by using Lemma \ref{lemma} with  $\beta\leq \alpha$,  the functional $\mathcal H$  is well defined for initial conditions $u_0\in H^2(\R^-)$ and homogeneous boundary conditions. Therefore,  $\mathcal H$  has the following monotonicity property
\begin{equation}\label{upper-bound-Lypunov}
\mathcal{H}[u](t)\leq \mathcal{H}[u](0),\;\, \text{for all}\;\,t>0.
\end{equation}

\begin{rem}[About breather's parameters]
	The condition $\beta \leq \alpha$ is consistent with the first hypothesis $\beta \leq \sqrt 3 \alpha$ 
	imposed in order to treat the case of mKdV breathers \textit{moving leftwards}. However, 
	we can not use \eqref{upper-bound-Lypunov} in the case $\alpha < \beta \leq \sqrt 3 \alpha$ because we do not
	control the right sign in $\mathcal{H}$ \eqref{Hlyapunov}, a contradiction with the energy growth. 
	In fact,  in this interval, the stability question remains open.
\end{rem}

Let $z \in H^{2}(\mathbb{R}^-),$ and $B=B_{\alpha, \beta}$ be any restricted mKdV breather \eqref{semibreather}. 
We define, the corresponding restriction to $\R^{-}$ of the quadratic form associated to $\mathcal{L}$ (see \eqref{quadratic}):
\begin{equation}\label{quadraticRest}
\begin{aligned}
\mathcal{Q}[z]&:=\int_{\mathbb{R}^-} z \mathcal{L}[z]\\
& = \int_{\mathbb{R}^-} z_{x x}^{2}+2\left(\beta^{2}-\alpha^{2}\right) \int_{\mathbb{R}^-} z_{x}^{2}+\left(\alpha^{2}+\beta^{2}\right)^{2} \int_{\mathbb{R}^-} z^{2}-5 \int_{\mathbb{R}^-} B^{2} z_{x}^{2}\\
&\hspace{2cm} + 5 \int_{\mathbb{R}^-} B_{x}^{2} z^{2}+10 \int_{\mathbb{R}^-} B B_{x x} z^{2}+\frac{15}{2} \int_{\mathbb{R}^-} B^{4} z^{2}-6\left(\beta^{2}
-\alpha^{2}\right) \int_{\mathbb{R}^-} B^{2} z^{2}.
\end{aligned}
\end{equation}

\medskip 
Now, in the spirit of \cite{AM} we have the following result.

\begin{lemma}\label{pertubation} Let $z \in H^{2}(\mathbb{R}^-)$ be any function with sufficiently small $H^{2}$-norm, and $B=B_{\alpha, \beta}$ be any breather function \eqref{semibreather}. Then, for all $t \in \mathbb{R},$ one has that $\mathcal H$ \eqref{Hlyapunov} verifies

\begin{multline}\label{pertubation-a}
\mathcal{H}[B+z]-\mathcal{H}[B]=\frac{1}{2} \mathcal{Q}[z]+\mathcal{N}[z] + B_{xx}(x=0,t)z_x(x=0,t)- B_{3x}(x=0,t)z(x=0,t)\\ 
-5B^2B_{x}(x=0,t)z(x=0,t) + B_{x}(x=0,t)z(x=0,t),
\end{multline}
with $\mathcal{Q}$ being the quadratic form defined in \eqref{quadraticRest} and $\mathcal{N}[z]$ satisfying $|\mathcal{N}[z]| \leq K\|z\|_{H^{2}(\mathbb{R}^-)}^{3}.$
\end{lemma}
\begin{proof} Just following \cite[Lemma 5.2]{AM}, we skip the details. Namely, 
expanding $\mathcal{H}[B+z]-\mathcal{H}[B]$ and collecting terms proportional to $z$, the only difference is that some trace 
terms appear as a consequence of the integration by parts. Indeed, 

\medskip 
	$$\frac12 \int_{\R^-} {2}B_{xx}z_{xx}=\int_{\mathbb{R}^-}B_{4x}z + B_{xx}(x=0,t)z_x(x=0,t)- B_{3x}(x=0,t)z(x=0,t),$$ 
	
	$$-\frac52 \int_{\R^-} 2B^2B_{x}z_{x}= -\frac52 \int_{\R^-} (-2B^2B_{xx}z - 4BB_x^2z) -5B^2B_{x}(x=0,t)z(x=0,t),$$
	
	\medskip 
	\noindent
	and
	
	$$\frac12 \int_{\R^-}2B_{x}z_{x}=\frac12 \int_{\R^-}-2B_{xx}z + B_{x}(x=0,t)z(x=0,t).$$
	
	\medskip 
	 Notice that the restricted breather $B$ \eqref{semibreather} also satisfies the differential identities given in Lemma  \ref{equations}, and hence the following fundamental identity
	\[G[B]:=B_{(4 x)}-2\left(\beta^{2}-\alpha^{2}\right)\left(B_{x x}+B^{3}\right)+\left(\alpha^{2}+\beta^{2}\right)^{2} B+5 B B_{x}^{2}+5 B^{2} B_{x x}+\frac{3}{2} B^{5}=0,\] 
	\noindent
    	it was used in the above expansion (see  Lemma \ref{equations}-{\it(iv)} for details).
\end{proof}

\section{Proof of Theorem \ref{teorema}}\label{ProofTheo}

The proof follows some ideas developed in \cite{AM} and \cite{CM}. These ideas allow us to avoid some problems caused by the trace terms $\tau_M(0,t)$, $\tau_E(0,t)$ and $\tau_F(0,t)$. In our proof, we adapted these previous arguments to the restricted breather $B$. The control of the shift function $\rho_2$, obtained in Lemma \ref{lemma-modulation}, will be a key step in the proof.

\subsection{Starting of the proof of Theorem \ref{teorema}}

Take $\alpha$ and $\beta$ satisfying $0< \beta\leq \alpha$ and fix  $L>L_0$, where $L_0$ will be taken larger enough. Assume that 
\begin{equation}\label{initial2}
\left\|u_0 - B(\cdot, 0 ; 0,L)\right\|_{H^{2}(\R^-)} \leq \eta
\end{equation}
is satisfied for $u_0$ and for $\eta\leq \eta_0$  with $\eta_0$ small enough to be chosen later.

\medskip
Let $u(\cdot, t) \in \mathcal{C}\big( \R^+, H^2(\R^-)\big)$ be the associated solution of the IBVP \eqref{IBVP_left} with initial data $u(x,0)=u_0$ and homogeneous boundary 
conditions.   By using the continuity of the flow (see Remark \ref{rem-wp}), given $\eta>0$ there exist a small time $T_0$ and continuous parameter functions $\rho_j(t)\, (j=1,2)$ such that 
\begin{equation}\label{T0-initial-control}
\sup\limits_{0 \le t \le T_0} \|u(\cdot, t) - B(\cdot, t, \rho_1(t), \rho_2(t) +L)\|_{H^2(\R^-)} \le 2\eta
\end{equation}
for all $0\le t\le T_0$. 

\medskip 
Let $K_0>2$ a constant  to be fixed later and consider the \textit{maximal time of stability}, defined as follows:
\begin{multline}\label{maximal-time-est}
T_*:=
\sup\bigg\{ T>0:\; \text{for all}\;  t\in [0, T] \;\text{there exist}\; \rho_1(t), \rho_2(t)\in \R\; \text{such that}
\\
\sup\limits_{0\le t \le T}\|u(\cdot, t) -B(\cdot, t, \rho_1(t), \rho_2(t) +L)\|_{H^2(\R^-)}\le  K_0 
(\eta + e^{-\beta L/2}) \bigg\}.
\end{multline}
Notice that from \eqref{T0-initial-control} we have that $T_*$ is well-defined. 

\medskip 
By choosing $L$ and $K_0$ large, with $\eta\leq \eta_0$, we will prove that $T_*=\infty$. The idea is to use a contradiction argument under the assumption 
$T_*<\infty$. Indeed, as we will see, a bootstrap type argument will ensure the inequality
\begin{equation}\label{property1}
\|u(\cdot, t) -B(\cdot, t, \rho_1(t), \rho_2(t) +L)\|_{H^2(\R^-)}\le  \frac{1}{2}K_0 (\eta + e^{-\beta L/2}),
\end{equation}
for all $0\le t \leq T_*$, which is a contradiction with the definition of $T_*$ (if it is finite).

\medskip 
We split the proof of \eqref{property1} in the following steps: first of all in subsection \ref{sub-modulation} we establish the modulation theory and 
exponential decays for the modulated breather in the boundary. 
Next, in subsection \ref{sub-error-estimate} we give some error estimates for the evolution in 
time of the  restricted Lyapunov functional \eqref{Hlyapunov}. Finally, in subsection \ref{sub-end-proof}, we derive 
the desired inequality \eqref{property1} to complete the proof.

\subsection{Modulation}\label{sub-modulation}

Using the notation introduced in \eqref{breather-simp-not}-\eqref{B1-B2} we have the following result.

\begin{lemma}\label{lemma-modulation}
 Let $T_{*}$ defined in \eqref{maximal-time-est}. There exist constants, $\eta_{0}>0$ small enough and $L_0$ large enough such that, for all $\eta \in\left(0, \eta_{0}\right)$ and $L>L_0$, the following holds. There exist continuous functions $\rho_{1}: [0, T_*] \to \R$ and $\rho_{2}:[0, T_*] \to  \left(-\frac{L}{2}, \frac{L}{2}\right)$, such that 

\begin{equation}\label{z}
z(x,t):=u(x,t)-B(x,t, \rho_1(t), \rho_2(t) +L) 
\end{equation}
satisfies the orthogonality conditions
\begin{equation}\label{ort-cond}
\int_{\R^-}B_j\left(x,t ; \rho_{1}(t), \rho_{2}(t)+L\right) z(x,t)dx=0, \quad j=1,2,
\end{equation}
for all $t \in [0, T_*]$. Moreover, there exist a positive constant $K>0$, independent of $K_0$, ensuring the following estimates:
\begin{align}
&\|z(\cdot, t)\|_{H^{2}(\mathbb{R}^-)} \leq K K_0 (\eta+e^{-\beta L/2}),\label{rho}\\ 
&\|z(\cdot, 0)\|_{H^{2}(\mathbb{R}^-)} \leq K( \eta+e^{-\beta L/2   }).
\end{align}
\end{lemma}

\begin{proof}
Let $K_0$ and $T_*$ as defined in \eqref{maximal-time-est}. We first define the set 
\begin{equation}\label{initial-neighborhood-int}
\mathcal{V}_t[K_0]:=\left\{ v\in H^2(\R^-): \inf_{\rho_1,\rho_2\in \R}\|v - B(\cdot, t, \rho_1,\rho_2  + L)\|_{H^2(\R^-)}\leq K_0 	
(\eta + e^{-\beta L/2})\right\}
\end{equation}
and we note that 
\begin{equation}\label{initial-neighborhood-1}
\inf_{\rho_1,\rho_2\in \R}\|u(\cdot, t) - B(\cdot, t, \rho_1,\rho_2  + L)\|_{H^2(\R^-)}\leq K_0 
(\eta + e^{-\beta L/2})
\end{equation}
for all $0\le t \leq T_*$ with $K_0$ large enough.  Hence, 
\begin{equation}\label{initial-neighborhood-2}
u(\cdot, t) \in \mathcal{V}_t[K_0],\quad 0\le t \le T_*.
\end{equation}

\medskip 
The idea is to apply the Implicit Function Theorem. Firstly, we define the functional operator:
$$\mathcal{J}=(\mathcal{J}_1, \mathcal{J}_2): H^2(\R^-)\times\R\times \R\longrightarrow  \R^2,\; j=1,2,$$
with
\be\label{functional-components}
\mathcal J_j[v; \rho_1, \rho_{2}]:=\int_{\R^-}\big(v( x)-B(x,t ; \rho_1, \rho_2 +L)\big) B_j(x,t ; \rho_1, \rho_2 +L) d x.
\ee
We can check that  $\mathcal{J}_j\, (j=1,2)$ are of class $\mathcal{C}^1$ and also satisfy 
\begin{equation}\label{J-functionals-zero}
\mathcal{J}_j[B(\cdot ,t ; \rho_1, \rho_2 +L); \rho_1, \rho_2]=0,
\end{equation}
for all $\rho_1, \rho_2 \in \R$. In what follows we use the notation $\partial_k:=\partial_{\rho_{\!_k}}$   and $\partial_{kj}:=\partial^2_{\rho_{\!_k}\rho_{\!_j}}$. So, 
for $j,k=1,2$, one has 
\begin{multline}
\partial_k \mathcal{J}_j[v;\rho_1,\rho_2]=-\int_{\R^-} B_k(x,t,\rho_1,\rho_2 +L)B_j(x,t,\rho_1,\rho_2 +L)dx \; +\\ \int_{\R^-}\big(v-B(x,t,\rho_1,\rho_2 +L)\big)\partial_{kj}B(x,t,\rho_1,\rho_2 +L) dx.
\end{multline}
Hence, we have
\begin{equation}\label{calc1}
\frak{J}_{jk}:=\partial_k \mathcal{J}_j[v;\rho_1,\rho_2]\bigg|_{v=B(\cdot,t;0,L)}=-\int_{\R^-} B_k(x,t;0,L)B_j(x,t;0,L)dx.
\end{equation}
and we define $\frak{J}$ as the $2 \times 2$ matrix with components 
\begin{equation}\label{matrix-J}
\frak{J} = \big(\frak{J}_{jk}\big)_{j, k=1,2}.
\end{equation}

\medskip
As in \cite{AM}, putting $B_j(x,t):=B_j(x,t;0,L)$ we have from Cauchy-Schwarz inequality and the fact that $B_1$ and $B_2$ are not parallel for all time that
$$
\operatorname{det} \frak{J}=-\left[\int_{\R^-} B_1^2(x,t)dx\int_{\R^-} B_2^2(x,t)dx-\left(\int_{\R^-} B_1(x,t)B_2(x,t) dx\right)^{2}\right](t ; 0, L) \neq 0
$$  
for all $0\le t \leq T_*$. 

\medskip 
 Therefore, in a small neighbourhood $U_t\times I_t \times J_t\subset H^{2}(\R^-)\times \R \times \R$ of the point $\big( B(t;0,0,L),0,0\big)$, and for $t \in\left[0, T_{*}\right]$ (given by the definition of \eqref{maximal-time-est}), it is possible to write the decomposition \eqref{z} satisfying 
 
 \begin{equation}\label{J}\mathcal{J}[u(\cdot, t), \rho_1(t;u(\cdot, t)), \rho_2(t;u(\cdot, t))]=0,\quad 0\le t\leq T_*\end{equation}  
 \noindent
 for $\eta_0$  small enough, $L$  larger enough  and for unique functions 
 $$\rho_1:=\rho_1(t,u(\cdot,t))\in I_t\quad \text{and}\quad \rho_2:=\rho_2(t,u(\cdot,t))\in J_t\subset\left(-\tfrac{L}{2},\tfrac{L}{2}\right).$$
 \noindent
 This directly implies that $\rho_2(t)>-\frac{L}{2}$. We choose this in order to control the traces of the modulated breather (see Lemma \ref{boundary}). The uniqueness of the functions $\rho_1$ and $\rho_2$ is a consequence of the uniqueness coming from the  Implicit Function Theorem in each $U_t \times I_t\times J_t$.  
\end{proof}

The following result shows an estimate for the trace terms of the breather solution 
$B$ which is localized far away from the origin (i.e at distance $L$).

\begin{cor}[Boundary values of $B$]\label{boundary}
Let $B=B(x,t;\rho_1(t),\rho_2(t)+L)$ a restricted breather given by \eqref{semibreather}. Then the following estimate holds:
\begin{equation}\label{trace1}
\big|\partial_j\partial_x^kB(0,t;\rho_1(t),\rho_2(t)+L))\big|\leq C e^{-\beta L/2},\quad j=1,2,~~k=0,1,2,3. 
\end{equation}
\end{cor}

\begin{proof}
 See Appendix \ref{Appboundary} for the proof of this result.
\end{proof}

\subsection{Error estimate}\label{sub-error-estimate}
Applying Lemma \ref{pertubation} to the solution $u\in \mathcal{C}(\R^+;H^2(\R^-))$ with homogeneous boundary conditions $u(0,t)=\partial_x u(0,t)=0$ 
and by using the smallness of $z(x,t):=u(x,t)-B(x,t)$ in \eqref{z}, we get
\begin{equation}
\begin{split}
\mathcal{H}[u](t)=&\mathcal{H}[B](t)+\frac{1}{2} \mathcal{Q}[z](t)+N[z](t)v\\
&\hspace{1.2cm}+ B_{xx}(x=0,t)z_x(x=0,t)- B_{3x}(x=0,t)z(x=0,t)\\ 
&\hspace{2.5cm}-5B^2B_{x}(x=0,t)z(x=0,t) + B_{x}(x=0,t)z(x=0,t).
\end{split}
\end{equation}

Now, by using that $u(0,t)=u_x(0,t)\equiv0$ we get 
\begin{equation}
\begin{split}
B_{xx}(x=0,t)&z_x(x=0,t)- B_{3x}(x=0,t)z(x=0,t)\\ 
&-5B^2B_{x}(x=0,t)z(x=0,t) + B_{x}(x=0,t)z(x=0,t)\\
&\hspace{1cm} \boldsymbol{=} -B_{xx}(x=0,t)B_x(x=0,t)+ B_{3x}(x=0,t)B(x=0,t)\\ 
&\hspace{2.5cm} +5B^2B_{x}(x=0,t)B(x=0,t) - B_{x}(x=0,t)B(x=0,t).
\end{split}
\end{equation}

We now fix the following notation: $\widetilde{\mathcal{H}}$ is the extension of $\mathcal{H}$ \eqref{Hlyapunov} to the whole line $\R$.

\begin{lemma}\label{error1}
	Let $B(x,t;\rho_1(t),\rho_2(t)+L)$  given by \eqref{breather}. Then the following error estimate holds
	
	\begin{equation}
	|\mathcal H[B](t)-\mathcal H[B](0)|\lesssim e^{-\beta L/2}.
	\end{equation}
	\end{lemma}
\begin{proof}
As introduced above, we have that $\widetilde{\mathcal {H}}[\widetilde B](t)=\widetilde{\mathcal {H}}[\widetilde B](0)$. By using the localization on  the left size of the breather, far away from the origin, we get the result.
\end{proof}

Now we continue with the proof. By using Lemmas \ref{pertubation} and \ref{error1} and Corollary  \ref{boundary}, we get for $t\le T_{*}$ that
\begin{equation}\label{quad1}
\begin{split}
\mathcal Q[z](t)&\leq C \mathcal Q[z](0)+K\|z(t)\|_{H^2(\R^-)}^3+K\|z(0)\|_{H^2(\R^-)}^3+Ke^{-\beta L/2}\\
&\leq C \|z(0)\|_{H^2(\R^-)}^2+C\|z(t)\|_{H^2(\R^-)}^3+Ke^{-\beta L/2}\\
&\leq C \eta^2+CK_0^3 (\eta +e^{-{\beta} L})^3+Ce^{-\beta L/2},
\end{split}
\end{equation}
where the term $\|z(t)\|_{H^2(\R^-)}^3$ was absorbed by $\|z(t)\|_{H^2(\R^-)}^2$.
\subsection{End of the proof of Theorem \ref{teorema}}\label{sub-end-proof}

The final step in the proof of Theorem \ref{teorema} consists of making a suitable extension of the required functions and functionals to the whole line.

\begin{defn} (Zero extension, left half-line case). Let $v \in H^{2}\left(\mathbb{R}^{-}\right)$ such that $v(x=0)=0$ and $v_x(x=0)=0$. 
We define its (zero) extension $\tilde{v}$ as the function
\begin{equation}\label{zero-extension}
\tilde{v}(x):=\left\{\begin{array}{ll}
v(x), & x \leq 0, \medskip \\
0,    & x>0.
\end{array}\right.
\end{equation}
\end{defn}

Note that $\widetilde B$
cannot be considered as the zero extension of $B$, since this function and its derivative does not vanish at the origin. Therefore we consider the 
natural extension of the breather $\widetilde B$ as given in \eqref{breather}. This interesting
difference will be important for the stability proof.

\medskip
 Let $\tilde u$ the extension of the solution $u$ defined in \eqref{zero-extension} and  consider the function \newline $\tilde z\in \mathcal{C}([0,\infty); H^2(\R))$ by
\begin{equation}\label{tildez}
\tilde z=\tilde u-\widetilde B.
\end{equation}
\noindent
Then, we write
\[\mathcal E[\tilde z] = \widetilde{\mathcal{ Q}}[\tilde z(t)]- \mathcal{Q}[z(t)],\]
\noindent
where $\mathcal E$ is the quadratic error functional restricted to $\R^+$, namely
\begin{equation}
\begin{split}
\mathcal{E}[\tilde z]:=&\int_{\mathbb{R}^+} z \mathcal{L}[z]= \int_{\mathbb{R}^+} z_{x x}^{2}+2\left(\beta^{2}-\alpha^{2}\right) \int_{\mathbb{R}^+} z_{x}^{2}+\left(\alpha^{2}+\beta^{2}\right)^{2} \int_{\mathbb{R}^+} z^{2}-5 \int_{\mathbb{R}^+} B^{2} z_{x}^{2} \\
&+5 \int_{\mathbb{R}^+} B_{x}^{2} z^{2}+10 \int_{\mathbb{R}^+} B B_{x x} z^{2}+\frac{15}{2} {\int_{\mathbb{R}^+}}B^{4} z^{2}-6\left(\beta^{2}-\alpha^{2}\right) \int_{\mathbb{R}^+} B^{2} z^{2}.
\end{split}
\end{equation}

\medskip 
From the above definitions, we have the following result on the error control.

\begin{lemma}\label{error} Let $\tilde z$ given by \eqref{tildez}. Then for any $t>0$ 
$$\mathcal E[\tilde z]\lesssim 	 e^{-\beta L/2}.$$
\end{lemma}	
\begin{proof}
Follows directly from the fact 	that
$\|\tilde z(t)\|_{H^2(\R^+)}\lesssim e^{-\beta L/2}$.
\end{proof}

Now we are able to treat the term $\mathcal Q[z(t)]$. To proceed, we use Proposition \ref{coec} and Lemma \ref{error} to get
\begin{equation}\label{quad}
\begin{split}
\mathcal Q[ z (t)]&=\widetilde{\mathcal Q}[\tilde z (t)]-\mathcal E[\tilde z]\geq \mu_0\|\tilde z\|_{H^2(\R)}^2-\frac{1}{\mu_{0}}\left(\int_{\mathbb{R}} \tilde z \tilde B\right)^{2}-Ce^{-\beta L/2}\\
&=\mu_0\|z(t)\|_{H^2(\R^+)}^2+\|z(t)\|_{H^2(\R^-)}^2-\frac{1}{\mu_{0}}\left(\int_{\R^-}\tilde z \tilde B\right)^{2}-Ce^{-\beta L/2}\\
&\geq \mu_0\|z(t)\|_{H^2(\R^-)}^2-\frac{1}{\mu_{0}}\left(\int_{\R^-}\tilde z \tilde B\right)^{2}-Ce^{-\beta L/2}.
\end{split}
\end{equation}
Now by using the conservation of the mass \eqref{mass-ident} we get
\begin{equation}
\begin{split}
\|u(t)\|_{L^2(\R^-)}^2&=\|(B+z)(t)\|_{L^2(\R^-)}^2=\|B(t)\|_{L^2(\R^-)}^2+\|z(t)\|_{L^2(\R^-)}^2+2\int_{\R^-}B(t)z(t)dx\\
&=\|B(0)\|_{L^2(\R^-)}^2+\|z(0)\|_{L^2(\R^-)}^2 + 2\int_{\R^-}B(0)z(0)dx=\|u_0\|_{L^2(\R^-)}.
\end{split}
\end{equation}
It follows that
\begin{equation}\label{tec4}
\begin{split}
\left|\int_{\R^-}B(t)z(t)dx\right|&\leq C\left|\int_{\R^-}B(0)z(0)dx\right| + C\|z(t)\|_{L^2(\R^-)}^2 + \|z(0)\|_{L^2(\R^-)}^2 + Ce^{-\beta L/2}\\
&\leq C \left(\eta+K_0^2 (\eta +e^{-{\beta} L/2})^2+\eta^2\right)+Ce^{-\beta L/2},
\end{split}
\end{equation}
for any $t\in [0,T_{*}]$. Replacing \eqref{tec4} in \eqref{quad1} we get 
\begin{equation*}
\mu_0\|z(t)\|_{H^2(\R^-)}^2\leq \frac{C}{\mu_0} \left(\eta+K_0 (\eta +e^{-{\beta} L/2})+\eta^2\right)^2+C \eta^2+C(K_0 (\eta +e^{-{\beta} L/2})^3+Ce^{-\beta L/2}.
\end{equation*}

Taking $K_0$ large  enough and $\eta$ small enough, we finally get 
\begin{equation}
\|z(t)\|_{H^2(\R^-)}^2\leq\frac12 K_0 (\eta +e^{-{\beta} L})^2
\end{equation}
as we stated in \eqref{property1}. Then, the  proof is finished.

\appendix
\section{Proof of Lemma \ref{lemma}}\label{proof-lemma}

	Multiplying the equation of the IBVP \eqref{IBVP_left} and integrating by parts one get 
\begin{equation}\label{proof-M-01}
	\begin{split}
		\frac{1}{2}\frac{d}{dt}\int_{\R^-}|u|^2dx  & = \int_{\R^-}u_{xx}u_xdx -[uu_{xx}](0,t) -3\int_{\R^-}u^3u_xdx\\
		&=\big[\tfrac{1}{2}u^2_x -u_{xx}u - \tfrac{3}{4}u^4\big](0,t)\\
		&=\tau_M(0,t),
	\end{split}
\end{equation}
which implies \eqref{mass-ident}  by integration in time of \eqref{proof-M-01}. Under hypotheses in \eqref{lemma-h-condit} we see that $\tau_M(0,t)=0$ and then we have \eqref{mass-ident-h}.

\medskip 
Now we derive with respect to $x$ the equation of the IBVP \eqref{IBVP_left} and after that we multiply by $u_x$ to get 
\begin{equation}\label{proof-E-01}
	\begin{split}
		\frac{1}{2}\frac{d}{dt}\int_{\R^-}u^2_xdx & = -\int_{\R^-}\big(u_{xxx} + (u^3)_x\big)_xu_xdx\\ 
		&= \int_{\R^-}u_{xxx}u_{xx}dx  + \int_{\R^-}(u^3)_xu_{xx}dx - \big[u_{xxx}u_x + 3u^2u^2_x\big](0, t)\\
		&=\int_{\R^-}u^3\big( u_t + (u^3)_x\big)dx + \big[ \tfrac{1}{2}u^2_{xx} + u^3u_{xx} -u_{xxx}u_x - 3u^2u^2_x\big](0, t)\\
		&=\frac{1}{4}\frac{d}{dt}\int_{-\infty}^0u^4dx + \big[ \tfrac{1}{2}u^6 + \tfrac{1}{2}u^2_{xx} + u^3u_{xx} -u_{xxx}u_x - 3u^2u^2_x\big](0, t). 
	\end{split}
\end{equation}
Thus,
\begin{equation}\label{proof-E-02}
	\frac{d}{dt}E[u](t)=\tau_E(0, t)
\end{equation}
which give us \eqref{energy-ident}. In the case of the homogeneous boundary condition \eqref{lemma-h-condit},  we see that 
$$\tau_E(0, t)=\tfrac{1}{2}u^2_{xx}$$
\noindent
and hence  the lower bound for the energy \eqref{energy-ident-h} is obtained.

\medskip

Now we are going to prove the identity \eqref{second-energy-ident}. We compute separately, the derivative in time for 
the integral terms $\displaystyle \int_{\R^-}u_{xx}^2dx$\,  and\,  $\displaystyle \int_{\R^-}u^6dx$. 
By using the structure of the mKdV equation  and integration by parts we get:
\begin{equation}\label{proof-F-01}
	\begin{split}
		\frac{1}{2}\frac{d}{dt}\int_{\R^-}&u_{xx}^2 dx = \int_{\R^-}u_{xx}u_{xxt}dx =-\int_{\R^-}u_{xxx}u_{xt}dx + [u_{xx}u_{xt}](0,t)\\
		&=\int_{\R^-}u_{xxx}\big[ u_{xxx} + (u^3)_x\big]_x dx + [u_{xx}u_{xt}](0,t)\\
		&= \frac{1}{2}u_{xxx}^2(0,t) + \int_{\R^-}u_{xxx}(u^3)_{xx}dx  + [u_{xx}u_{xt}](0,t)\\
		&=-\int_{\R^-}u_t(u^3)_{xx}dx - \int_{\R^-}(u^3)_x(u^3)_{xx}dx + \frac{1}{2}u_{xxx}^2(0,t) + [u_{xx}u_{xt}](0,t)\\
		&=\frac{3}{2}\int_{\R^-}u^2\big(u^2_x\big)_tdx + \Big[-u_t(u^3)_x - \frac{1}{2}\big((u^3)_x\big)^2 + \frac{1}{2}u_{xxx}^2 + u_{xx}u_{xt}\Big](0,t).
	\end{split}
\end{equation}

\medskip
On the other hand we have 
\begin{equation}\label{proof-F-2}
	\begin{split}
		\frac{1}{4}\frac{d}{dt}\int_{\R^-}u^6dx&= -\frac{3}{2}\int_{\R^-}u^5\big[u_{xxx}+ (u^3)_x\big]dx\\
		&=\frac{3}{2}\int_{\R^-}(u^5)_xu_{xx}dx -\frac{3}{2}u^5u_{xx}(0,t)-\frac{9}{2}\int_{-\infty}^0u^7u_xdx\\
		&=\underbrace{\frac{3}{2}\int_{\R^-}(u^5)_xu_{xx}dx}_{\boldsymbol{I}}-\frac{3}{2}[u^5u_{xx}](0,t) -\frac{9}{16}u^8(0,t)
	\end{split}
\end{equation}
with
\begin{equation*}
	\begin{split}
		\boldsymbol{I}&= -\frac{3}{2}\int_{\R^-}\big[u_t +u_{xxx}\big]u^2u_{xx}dx + \frac{3}{2}\int_{\R^-}u^3 (u^2)_xu_{xx}dx\\
		&= -\frac{3}{2}\int_{\R^-}u_tu^2u_{xx}dx-\frac{3}{2}\int_{\R^-}u_{xxx}u^2u_{xx}dx + \frac{2}{5}\,\boldsymbol{I}\\
		&= \frac{3}{2}\int_{\R^-}(u_tu^2)_xu_xdx -\frac{3}{2}[u_tu^2u_x](0,t) -\frac{3}{4}\int_{\R^-}u^2(u_{xx}^2)_xdx + \frac{2}{5}\,\boldsymbol{I};
	\end{split}
\end{equation*}
so
\begin{equation}\label{proof-F-3}
	\begin{split}
		\boldsymbol{I}&= \frac{5}{4}\int_{\R^-}u^2(u_x^2)_tdx + \frac{5}{2}\int_{\R^-}(u^2)_tu_x^2dx -\frac{5}{4}\int_{\R^-}u^2(u_{xx}^2)_xdx -\frac{5}{2}[u_tu^2u_x](0,t)\\
		&=\frac{5}{4}\int_{\R^-}u^2(u_x^2)_tdx + \frac{5}{2}\int_{\R^-}(u^2)_tu_x^2dx +\underbrace{\frac{5}{2}\int_{\R^-}uu_xu^2_{xx}dx}_{\boldsymbol{J}}\\
		&\hspace{7cm} - \frac{5}{4}[u^2u^2_{xx}](0,t) -\frac{5}{2}[u_tu^2u_x](0,t).
	\end{split}
\end{equation}
The integral $\boldsymbol{J}$ can be compute as follows:
\begin{equation*}
	\begin{split}
		\boldsymbol{J} &= \frac{5}{2}\int_{\R^-}uu_xu_{xx}u_{xx}dx\\
		&=-\frac{5}{2}\int_{\R^-}u_x^3u_{xx}dx - \boldsymbol{J} -\frac{5}{2}\int_{\R^-}uu_x^2u_{xxx}dx + \frac{5}{2}[uu_x^2u_{xx}](0,t);
	\end{split}
\end{equation*}	
thus,
\begin{equation*}
	\begin{split}
		\boldsymbol{J} & =  -\frac{5}{16}u_x^4(0,t) -\frac{5}{4}\int_{\R^-}uu_x^2u_{xxx}dx  + \frac{5}{4}[uu_x^2u_{xx}](0,t)\\
		&= \frac{5}{8}\int_{\R^-}(u^2)_tu_x^2dx + \frac{15}{4}\int_{\R^-}u^3u^3_xdx -\frac{5}{16}u_x^4(0,t) + \frac{5}{4}[uu_x^2u_{xx}](0,t).
	\end{split}
\end{equation*}	
It is not difficult to check that 
$$\int_{\R^-}u^3u^3_xdx= -\frac{1}{15}\boldsymbol{I} +  \frac{1}{4}[u^4u_x^2](0,t),$$
then we have 
\begin{equation*}
	\boldsymbol{J}=	\frac{5}{8}\int_{\R^-}(u^2)_tu_x^2dx -\frac{1}{4}\boldsymbol{I} + \frac{15}{16}[u^4u_x^2](0,t) -\frac{5}{16}u_x^4(0,t) + \frac{5}{4}[uu_x^2u_{xx}](0,t)
\end{equation*}
and putting the expression of $\boldsymbol{J}$ in \eqref{proof-F-3} one get
\begin{multline*}
	\boldsymbol{I}= \int_{\R^-}u^2(u_x^2)_tdx + \frac{5}{2}\int_{\R^-}(u^2)_tu_x^2dx\\
	+ \Big[-u^2u^2_{xx} -2u_tu^2u_x+ \frac{3}{4}u^4u_x^2 -\frac{1}{4}u_x^4 + uu_x^2u_{xx}\Big](0,t).
\end{multline*}
Now, by using the expression obtained for $\boldsymbol{I}$ and adding \eqref{proof-F-01} with \eqref{proof-F-2}  we get 
\begin{equation}\label{proof-F-4}
	\frac{1}{2}\frac{d}{dt}\int_{\R^-}u_{xx}^2dx +\frac{1}{4}\frac{d}{dt}\int_{\R^-}u^6dx=\frac{5}{2}\frac{d}{dt}\int_{\R^-}u^2u_x^2dx + \tau_F(0,t),
\end{equation}
with $\tau_F$ defined in \eqref{trace-terms-F}. 

\medskip 
Finally \eqref{second-energy-ident} is obtained by integrating  \eqref{proof-F-4} and under the homogeneous condition \eqref{lemma-h-condit} we have  
$\tau_F(0, s)=0$ that yields \eqref{second-energy-ident-h}.

\section{Proof of Corollary \ref{boundary}}\label{Appboundary}
The proof of \eqref{trace1} in the  case $k=0$ follows directly from \eqref{semibreather}  and the fact  that 
		\[\gamma t+\rho_2+L>\tfrac{L}{2}\quad \forall\ t\in [0,T_*],\]
	which is a consequence of $\rho_2(t)\in (-\frac{L}{2},\frac{L}{2})$ obtained in Lemma \ref{lemma-modulation}, and the hypothesis $\gamma>0$. Now we prove \eqref{trace1}. 
	By direct computation, we obtain
	\begin{equation}\label{tec0}
		\partial_x	B(x,t;\rho_1(t),\rho_2(t))=\frac{4 \sqrt{2} \alpha \beta h_{1}(x,t)}{\left(\alpha^{2}+\beta^{2}+\alpha^{2} \cosh \left(2 \beta y_{2}\right)-\beta^{2} \cos \left(2 \alpha y_{1}\right)\right)^{2}}
	\end{equation}
	where
	$$
	\begin{array}{l}
		h_{1}(x,t):=-\left(\alpha^{2}+\beta^{2}\right) \cosh \left(\beta y_{2}\right) \sin \left(\alpha y_{1}\right)\left[\alpha^{2}+\beta^{2}+\alpha^{2} \cosh \left(2 \beta y_{2}\right)-\beta^{2} \cos \left(2 \alpha y_{1}\right)\right] \\
		-2 \alpha \beta\left[\alpha \cos \left(\alpha y_{1}\right) \cosh \left(\beta y_{2}\right)-\beta \sin \left(\alpha y_{1}\right) \sinh \left(\beta y_{2}\right)\right]\left[\beta \sin \left(2 \alpha y_{1}\right)+\alpha \sinh \left(2 \beta y_{2}\right)\right]
	\end{array},
	$$ 
	where
	$y_1 := x+ \rho_1(t) + x_1, \  y_2 : = x+ \rho_2 (t) + x_2$ and $	\delta := \al^2 -3\bt^2, \quad  \ga := 3\al^2 -\bt^2$.
	Now, we control the function $h_1$ as follows
	\begin{equation}\label{tec1}
		\begin{split}
			|h_1(x,t)|&\leq C_{\alpha,\beta} \cosh (\beta y_2) [1+\cosh(2\beta y_2)]\\
			&\quad+C_{\alpha,\beta}[\cosh (\beta y_2)+1+|\sinh(\beta y_2)|][(1+|\sinh(2\beta y_2)|)]\\
			&\leq \cosh^3(\beta y_2)).
		\end{split}
	\end{equation}
	We also have that
	\begin{equation}\label{tec2}
		\frac{1}{\left(\alpha^{2}+\beta^{2}+\alpha^{2} \cosh \left(2 \beta y_{2}\right)-\beta^{2} \cos \left(2 \alpha y_{1}\right)\right)^{2}}\leq \frac{K_{\alpha,\beta}}{\cosh(2 \beta y_2)^2}.
	\end{equation}
	
	\noindent
	Thus, by using \eqref{tec1}, \eqref{tec2} in \eqref{tec0}
	\begin{equation}\label{tec3}
		|\partial_x	B(x,t;\rho_1(t),\rho_2(t))|\leq C_{\alpha,\beta}\sech \left(\beta y_{2}\right).
	\end{equation}
	Finally, combining \eqref{tec3} and the fact $\rho_2(t)\in (-\frac{L}{2},\frac{L}{2})$ we get the case $k=1$. The case $k=2$ follows from \eqref{Bt}. The case $k=3$ just follows from the PDE \eqref{mkdv}.

\end{document}